\documentclass[a4paper,11pt,makeidx]{amsart}
\usepackage{amscd}
\usepackage{xypic}  
\usepackage{amssymb}
\usepackage{amsthm}
\usepackage{epsf}
\makeindex
\newtheoremstyle{break}
  {9pt}
  {9pt}
  {\itshape}
  {}
  {\bfseries}
  {.}
  {\newline}
  {}
\theoremstyle{break}

\newtheorem*{thmA}{Theorem A}
\newtheorem*{thmB}{Theorem B}

\theoremstyle{plain}
\newtheorem{thm}{Theorem}[section]
\newtheorem{cor}[thm]{Corollary}

\newtheorem{lemma}[thm]{Lemma}

\newtheorem{prop}[thm]{Proposition}

\newtheorem{rem}[thm]{Remark}

\newtheorem*{ackn}{Acknowledgment}

\def\RR{{\mathbb R}}

\def\max{\operatorname{max}}

\def\Supp{\operatorname{Supp}}
\def\Exc{\operatorname{Exc}}

\def\trdeg{\operatorname{tr.deg}}

\def\NN{{\mathbb N}}

\def\PP{{\mathbb P}}
\def\B{\mathbf{B}}

\def\N{{\mathbf N}}

\def\*{\otimes}

\def\+{\oplus}                   
\def\*{\otimes}                  

\def\Supp{\operatorname{Supp}}

\def\Bs{\operatorname{Bs}}
\def\vol{\operatorname{vol}}

\newcommand{\ie}{{\rm i.e.\ }}

\begin{document}

\title[Augmented base loci and restricted volumes on normal varieties]{Augmented base loci and restricted volumes on normal varieties}

\author[S. Boucksom, S. Cacciola and A.F. Lopez]{S\'ebastien Boucksom*, Salvatore Cacciola** and Angelo Felice Lopez**}

\thanks{* Research partially supported by ANR projects MACK and POSITIVE.}

\thanks{** Research partially supported by the MIUR national project ``Geometria delle variet\`a algebriche" PRIN 2010-2011.}

\address{\hskip -.43cm CNRS-Universit{\'e} Pierre et Marie Curie, I.M.J., F-75251 Paris Cedex 05, France.
\newline e-mail {\tt boucksom@math.jussieu.fr}}

\address{\hskip -.43cm Dipartimento di Matematica e Fisica, Universit\`a di Roma
Tre, Largo San Leonardo Murialdo 1, 00146, Roma, Italy. \newline e-mail {\tt cacciola@mat.uniroma3.it, lopez@mat.uniroma3.it}}


\begin{abstract}
We extend to normal projective varieties defined over an arbitrary algebraically closed field a result of Ein, Lazarsfeld, Musta{\c{t}}{\u{a}}, Nakamaye and Popa characterizing the augmented base locus (aka non-ample locus) of a line bundle on a smooth projective complex variety as the union of subvarieties on which the restricted volume vanishes. We also give a proof of the folklore fact that the complement of the augmented base locus is the largest open subset on which the Kodaira map defined by large and divisible multiples of the line bundle is an isomorphism. 
\end{abstract}

\maketitle

\section{Introduction}
\label{intro}
We work over an arbitrary algebraically closed field. The \emph{stable base locus} of a line bundle $L$ on a projective variety $X$ is the Zariski closed subset defined as
$$
\B(L):=\bigcap_{m\in \NN} \Bs\left(mL\right), 
$$
and the \emph{augmented base locus} (aka \emph{non-ample locus}) of $L$ is 
$$
\B_+(L):=\bigcap_{m\in \NN} \B\left(mL-A\right),
$$
where $A$ is any ample line bundle on $X$. This construction appears several times over in the literature (notably in \cite{n}), and was formally introduced in \cite[Def. 1.2]{elmnp1} (see also \cite[Def. 3.16]{b2} for its analytic counterpart for $(1,1)$-classes). Given its basic nature, it naturally plays an important role in birational geometry, as illustrated by \cite{t,hm,bdpp,bchm}, to mention only a few. 

Trivially, $\B_+(L)$ is empty iff $L$ is ample, while $\B_+(L)\ne X$ iff $L$ is big. In that case, the Kodaira map
$$
\Phi_m:X\dashrightarrow\PP H^0(X,mL)
$$
defined by the sections of $mL$ is birational onto its image for all $m$ sufficiently large and divisible. Our first main result is the following, which seems to be a folklore fact in the subject: 
\begin{thmA} Let $L$ be a big line bundle on a normal projective variety $X$. Then the complement $X\setminus\B_+(L)$ of the augmented base locus is the largest Zariski open subset $U\subseteq X\setminus\B(L)$ such that, for all large and divisible $m$, the restriction of the morphism
$$
\Phi_m:X\setminus\B(L)\to\PP H^0(X,mL)
$$ 
to $U$ is an isomorphism onto its image. 
\end{thmA}

For every subvariety $Z\subseteq X$ not contained in $\B_+(L)$, the restriction of $L$ to $Z$ is big. Better still, the space of sections of $mL|_Z$ that extend to $X$ has maximal possible growth: if we denote by $H^0(X|Z,mL)$  the image of the restriction map $H^0(X,mL)\to H^0(Z,mL_{|Z})$ and set $d:=\dim Z$, then the \emph{restricted volume} of $L$ on $Z$, introduced in \cite{elmnp1} and defined as 
$$
\vol_{X|Z}(L) = \limsup\limits_{m \to + \infty} \frac{d!}{m^d}\dim H^0(X|Z, mL),
$$
is positive. In other words, we have
$$
\B_+(L)\supseteq\bigcup\limits_{\genfrac{}{}{0pt}{}{Z\subseteq X :}{\vol_{X|Z}(L)=0}} Z.
$$
Conversely, when $X$ is a smooth complex projective variety, \cite[Thm C]{elmnp2} states that $\vol_{X|Z}(L)=0$ for every irreducible component $Z$ of $\B_+(L)$, so that the above inclusion is an equality. The proof of this result is quite involved, using the whole arsenal of asymptotic invariants (jet separation, Hilbert-Samuel and Arnold multiplicities, etc...) and a delicate combination of Fujita approximation arguments with estimates for spaces of sections. 

The goal of the present paper is to provide an elementary proof of this result, valid furthermore for any normal projective variety over an arbitrary algebraically closed field.  

\begin{thmB} Let $X$ be a normal projective variety defined over an arbitrary algebraically closed field. For every line bundle $L$ on $X$ and every irreducible component $Z$ of $\B_+(L)$ we have $\vol_{X|Z}(L)=0$, and hence
$$
\B_+(L)=\bigcup\limits_{\genfrac{}{}{0pt}{}{Z\subseteq X :}{\vol_{X|Z}(L)=0}} Z.
$$
\end{thmB}

It is important to emphasize that the difficult original proof given in \cite{elmnp2} is actually valid for $\RR$-divisors, and that it yields a much stronger continuity result: if $Z$ is an irreducible component of $\B_+(L)$ and $A$ is ample, then 
\begin{equation}\label{equ:limvol}
\lim_{m\to+\infty}\vol_{X|Z}\left(L+\tfrac 1 m A\right)=0. 
\end{equation}
Already for $Z=X$, \ie when $L$ is not big, $\vol_{X|X}(L)=\vol(L)$ is zero just by definition, while $\lim_{m\to\infty}\vol\left(L+\tfrac 1 m A\right)=0$ amounts to the continuity of the volume function \cite[Thm 2.2.37]{l} (see also \cite[Cor 4.11]{b3} for the case of $(1,1)$-classes). 

The stronger continuity statement (\ref{equ:limvol}) also seems to be needed to recover Nakamaye's original result in the nef case \cite{n}, which states that
$$
\B_+(L)=\bigcup\limits_{\genfrac{}{}{0pt}{}{Z\subseteq X :}{L^{\dim Z}\cdot Z=0}} Z
$$
when $L$ is nef. While we are not able to deduce this result from Theorem A, it was recently established in positive characteristic in \cite{ckm}, for $(1,1)$-classes in \cite{ct} and for arbitrary projective schemes over a field in \cite{bir}. 

\begin{ackn}
We wish to thank Lorenzo Di Biagio, Gianluca Pacienza and Paolo Cascini for some helpful discussions.
\end{ackn}

\section{An Iitaka-type estimate for graded linear series} \label{growth}
Throughout the paper we work over an arbitrary algebraically closed field $k$. An \emph{algebraic variety} is by definition an integral separated scheme of finite type over $k$. Let $L$ be a line bundle on a projective variety $Z$, and assume that its section ring
$$
R(Z,L):=\bigoplus_{m\in\N} H^0(Z,mL)
$$
is non-trivial, so that the semigroup
$$
\N(Z,L):=\left\{m\in\N\mid H^0(Z,mL)\ne 0\right\}
$$
is infinite. Assuming that $k$ has characteristic $0$, S. Iitaka proved in \cite{i} the existence of $C>0$ such that
$$
C^{-1}m^{\kappa(Z,L)}\le h^0(Z,mL)\le C m^{\kappa(Z,L)}
$$
for all $m\in\N(Z,L)$. Here $\kappa(Z,L)$ is an integer in $\{0,...,\dim Z\}$ known as the \emph{Iitaka dimension} of $L$, and which can also be characterized as 
$$
\kappa(Z,L)=\trdeg\left(R(Z,L)/k\right)-1. 
$$
In \cite{i}, the assumption that $k$ has characteristic zero is used to apply Hironaka's resolution of singularities and flattening theorems. 

Here we provide a simple geometric argument, directly inspired by the proof of \cite[Lem. 3.6]{dp}, proving the following more general result:

\begin{prop}
\label{prop:bigO} Let $L$ be a line bundle on a projective variety $Z$. Let $W \subseteq R(Z,L)$ be a non-trivial graded subalgebra, and set $\kappa(W):=\trdeg(W)-1$. Then there exists $C>0$ such that 
\[ C^{-1}m^{\kappa(W)}\le\dim W_m\le Cm^{\kappa(W)} \]
for all $m\in\N(W)$.
\end{prop}
We have set as usual
$$
\N(W)=\left\{m\in\N\mid W_m\ne 0\right\}.
$$
\begin{rem} Using the theory of Okounkov bodies, a much more precise estimate can actually be obtained. Indeed, it is proved in \cite[Thm 4]{kk} (see also \cite[Thm 0.1]{b}) that there exists $c\in(0,+\infty)$ such that
$$
\dim W_m=c\,m^{\kappa(W)}+o\left(m^{\kappa(W)}\right) 
$$
as $m\in\N(W)$ tends to $+\infty$. 
\end{rem}

The proof of Proposition \ref{prop:bigO} relies on the following standard facts. 
\begin{lemma}\label{lem:dimphi}
In the notation of Proposition \ref{prop:bigO}, let
$$
\Psi_m : Z \dashrightarrow \PP W_m 
$$
be the Kodaira map defined by the linear series $W_m$, with $m\in\N(W)$. Then
$$
\kappa(W) = \max\left\{ \dim \Psi_m(Z)\mid m \in \N(W)\right\}.
$$
\end{lemma}
\begin{proof} We recall the easy argument. Introduce the homogeneous fraction field 
$$
K_0(W) =\left\{f/g\mid f,g \in W_m \mbox{\ for some } m \right\}.
$$
The fraction field of $W$ is then a purely transcendental extension of degree one of $K_0(W)$, and hence $\kappa(W) = \trdeg(K_0(W)/k)$. Now $K_0(W)$, being a subfield of $K(Z)$, is also finitely generated over $k$, and it is thus the fraction field of the graded subalgebra spanned by $W_m$ for all $m\in\N(W)$ large enough. But the latter is the function field of $\Psi_m(Z)$, almost by definition.
\end{proof}

\begin{lemma}
\label{lem:dom} Let $f:Z\dashrightarrow Y$ be a dominant rational map between projective varieties such that $\dim Z>\dim Y$. Then any irreducible ample divisor $H$ of $Z$ dominates $Y$. 
\end{lemma}
Even though this fact probably sounds obvious, we provide a proof for completeness. 
\begin{proof} Let $Z'$ be the normalization of the graph of $f$, which comes with a birational morphism $\mu: Z'\to Z$ such that $g:=f\circ\mu:Z'\to Y$ is a morphism. Let also 
$$ 
\xymatrix{ Z' \ar^{f'}[r] & Y'  \ar^{\nu}[r]  & Y } 
$$
be the Stein factorization of $g$, so that $Y'$ is normal and $f'$ has connected fibers. Since $\nu$ is surjective, it is enough to show that $f'(\mu^*H)=Y'$. If $C'\subseteq Z'$ is a general curve contained in a general $f'$-fiber, then $C:=\mu(C')$ is also a curve, and hence $\mu^*H\cdot C'=H\cdot C>0$ by the projection formula and the ampleness of $H$. It follows that $\mu^*H$ meets the general fiber of $f'$, and hence $f'(\mu^*H)=Y'$. 
\end{proof}

\begin{proof}[Proof of Proposition \ref{prop:bigO}] The lower bound on $\dim W_m$ is a general property of graded integral domains, and is easy to get, choosing $\kappa(W)+1$ algebraically independent homogeneous elements of $W$. 

To get the upper bound, we first observe that the base field $k$ may be assumed to be uncountable, since the desired estimate is invariant under base field extension. By Lemma \ref{lem:dimphi}, we have $\dim  \Psi_m(Z) = \kappa(W)$ for all $m \in \N(W)$ large enough. If $\kappa(W) = \dim(Z)$ let $A$ be an ample divisor on $Z$ such that $L \leq A$. Then 
$$
\dim W_m \le h^0(mL) \le h^0(mA) \le  Cm^{\dim(Z)} = Cm^{\kappa(W)}.
$$ 
If $\kappa(W) < \dim(Z)$, Lemma \ref{lem:dom} allows to choose a complete intersection $T \subseteq Z$ of very ample divisors, very general in their linear series, in such a way that $\dim T = \kappa(W)$ and $\Psi_m(T)= \Psi_m(Z)$ for all $m\in\N(W)$ large enough. At this point, the uncountability of $k$ is used, since we impose countably many conditions. Now we claim that the restriction map 
$$
W_m \to H^0(T,mL_{|T})
$$ 
must be injective for all $m\in\N(W)$ large enough, which will provide the desired upper bound on $\dim W_m$. Indeed, if there is a non-zero section $s \in W_m$ vanishing along $T$ and with zero divisor $D$, then $\Psi_m(\Supp D) = \Psi_m(Z)$. But this gives a contradiction since $\Psi_m(\Supp D)$ is, by definition of $\Psi_m$, a hyperplane section of $\Psi_m(Z)$, while the latter is linearly non-degenerate in $\PP W_m$.
\end{proof}
The main consequence for us is:

\begin{cor}\label{cor:contracted} Let $L$ be a line bundle on a projective variety $X$, and assume that $Z\subset X$ is a positive dimensional subvariety not contained in the stable base locus $\B(L)$. Then $\vol_{X|Z}(L)=0$ iff for all $m$ large and divisible enough the Kodaira map $\Phi_m:X\dashrightarrow\PP H^0(X,mL)$ defined by the sections of $mL$ contracts $Z$, \ie
$$
\dim\Phi_m(Z)<\dim Z.
$$
\end{cor}
\begin{proof} Introduce the image $R(X|Z,L)\subseteq R(Z,L)$  of the restriction map $R(X,L)\to R(Z,L)$. To say that $Z$ is not contained in $\B(L)$ means that $R(X|Z,L)$ is non-trivial. For all $m$ large and divisible enough there is a commutative diagram
\begin{equation*}
\xymatrix{ \hskip .2cm Z_{\null} \ar@{^{(}->}^{i}[d] \ar@{-->}^{\hskip -1cm \Psi_m}[r] &  \hskip .1cm \PP H^0(X|Z,mL)_{\null}  \ar@{^{(}->}^{j} [d] \\ X \ar@{-->}_{\hskip -1cm \Phi_m}[r] &  \hskip .1cm \PP H^0(X,mL) }
\end{equation*}
in which both $i$ and $j$ are closed immersions. Then
$$
\kappa(X|Z,L)=\dim\Psi_m(Z)=\dim\Phi_m(Z)
$$
for all $m$ large and divisible enough. On the other hand, Proposition \ref{prop:bigO} implies that $\vol_{X|Z}(L)=0$ iff $\kappa(X|Z,L)<\dim Z$, and the result follows. 
\end{proof}

\section{Proof of the main theorems} 
\label{max}
\subsection{Proof of Theorem A}
Our main tool will be the following result from \cite{bbp} (whose proof is characteristic free). 

\begin{lemma}\cite[Proposition 2.3]{bbp}\label{lem:bbp} Let $\pi:X'\to X$ be a birational morphism between normal projective varieties. If $L$ is a big line bundle on $X$, then 
$$
\B_+(\pi^*L)=\pi^{-1}\left(\B_+(L)\right)\cup\Exc(\pi).
$$
\end{lemma}

One direction in Theorem A is almost trivial. Indeed, if $A$ is a very ample line bundle on $X$, there exists $m_0\in\N$ such that $\B(m_0L-A)=\B_+(L)$. It follows that $mm_0L-mA$ is base point free on $X\setminus\B_+(L)$ for all $m$ large and divisible enough, which implies that $\Phi_{mm_0}$ is an isomorphism on $X\setminus\B_+(L)$.

Conversely, pick $m$ large and divisible enough to ensure that $\Bs(mL)=\B(L)$ and that $\Phi_m$ is birational onto its image. Consider the commutative diagram
\begin{equation}
\label{diag}
\xymatrix{ \hskip .2cm X_m \ar_{\mu_m}[d] \ar^{f_m}[r] &  \hskip .1cm Y_m  \ar_{\nu_m}[d] \\ X \ar@{-->}_{\hskip -.2cm \Phi_m}[r] &  \hskip .1cm \Phi_m(X) }
\end{equation}
where $\mu_m$ is the normalized blow-up of $X$ along the base ideal of $mL$, $\nu_m$ is the normalization of $\Phi_m(X)$, and $f_m:X_m\to Y_m$ is the induced birational morphism between normal projective varieties. By construction, we have a decomposition 
$$
\mu_m^*(mL)=f_m^*A_m+F_m
$$
where $A_m$ is an ample line bundle on $Y_m$ and $F_m$ is an effective divisor with 
$$
\Supp F_m=\mu_m^{-1}\left(\B(L)\right). 
$$ 
If we denote by $U_m\subseteq X\setminus\B(L)$ the largest open subset on which $\Phi_m$ is an isomorphism, then $\nu_m\circ f_m$ is an isomorphism on $\mu_m^{-1}(U_m)$ since $\mu_m$ is an isomorphism over $X\setminus\B(L)$, and it follows that 
\begin{equation}\label{equ:um}
\mu_m^{-1}(U_m)\subseteq X_m\setminus\left(\Exc(f_m)\cup\Supp F_m)\right).
\end{equation}
Since $\mu_m(\Exc(\mu_m))$ is contained in $\B(L)\subseteq\B_+(L)$, Lemma \ref{lem:bbp} yields
$$
\B_+(\mu_m^{\ast}L) =  \mu_m^{-1}(\B_+(L)).
$$
On the other hand, we have
$$
\B_+(\mu_m^*L)=\B_+(\mu_m^*(mL))=\B_+(f_m^*A_m + F_m) \subseteq \B_+(f_m^*A_m) \cup \Supp(F_m).
$$
Another application of Lemma \ref{lem:bbp} thus shows that
$$
\mu_m^{-1}\left(\B_+(L)\right)\subseteq\Exc(f_m)\cup\Supp F_m,
$$
and we conclude as desired that $U_m\subset X\setminus\B_+(L)$, thanks to (\ref{equ:um}).

\subsection{Proof of Theorem B}\label{c}
We use the notation in the previous section. Let $Z$ be an irreducible component of $\B_+(L)$, so that $Z$ is necessarily positive dimensional by \cite[Proposition 1.1]{elmnp2} (which relies on a result of \cite{z} valid for normal varieties over any algebraically closed field). If $Z$ is contained in $\B(L)$, then we obviously have $\vol_{X|Z}(L)=0$ since $H^0(X|Z,mL)=0$ for all $m\ge 1$. We may thus assume that $Z$ is not contained in $\B(L)$; in view of Corollary \ref{cor:contracted}, we are to show that
$$
\dim\Phi_m(Z)<\dim Z
$$
for all $m$ large and divisible enough. The proof of Theorem A gives that
\begin{equation}
\mu_m^{-1}\left(\B_+(L)\right)=\mu_m^{-1}(X\setminus U_m)=\Exc(f_m) \cup \Supp F_m
\end{equation}
for all $m$ large and divisible enough, so that the strict transform $Z_m$ of $Z$ on $X_m$ is an irreducible component of $\Exc(f_m)$. Since $f_m$ is a birational morphism between normal varieties, it follows that $\dim f_m(Z_m)< \dim Z$. As $\nu_m$ is finite, it follows as desired that
$$
\dim\Phi_m(X)=\dim (\nu_m\circ f_m)(Z_m)<\dim Z.
$$

\end{document}